\documentclass[11pt]{amsart}

\usepackage[dvipsnames]{xcolor}
\usepackage[pagebackref,backref=true,colorlinks=true, linkcolor=Sepia, citecolor=Sepia]{hyperref}
\usepackage{amsmath,amsthm,amssymb,fancyhdr,graphicx,bbm,cancel,mathrsfs,todonotes,xypic,pinlabel}
\usepackage{tikz}
\usepackage{extarrows,tipa}
\usetikzlibrary{matrix}
\usepackage[all]{xy}
\usepackage{mathtools}
\usepackage{bm}
\usepackage{verbatim}
\usepackage{microtype}
\usepackage{subcaption}

\theoremstyle{theorem}
\newtheorem{thm}{Theorem}[section]

\newtheorem{theoremalpha}{Theorem}

\newtheorem{lem}[thm]{Lemma}
\newtheorem{prop}[thm]{Proposition}
\newtheorem{cor}[thm]{Corollary}
\theoremstyle{definition} 

\newtheorem{rmk}[thm]{Remark}
\theoremstyle{remark} 

\usepackage{enumerate,pdfpages,stmaryrd}
\usepackage[margin=1.19in, marginparwidth=1in]{geometry}
\usepackage{tikz}
\usepackage{dsfont}
\usetikzlibrary{matrix}

\renewcommand{\xto}{\xrightarrow}

\newcommand{\C}{\mathbb{C}}

\newcommand{\calC}{\mathcal{C}}
\newcommand{\Z}{\mathbb{Z}}
\newcommand{\Q}{\mathbb{Q}}
\newcommand{\R}{\mathbb{R}}

\newcommand{\Hyp}{\mathbb{H}}

\newcommand{\beq}{\begin{equation*}}
\newcommand{\eeq}{\end{equation*}}

\newcommand{\GL}{\mathrm{GL}}

\DeclareMathOperator{\Diff}{Diff}

\DeclareMathOperator{\Map}{Map}
\DeclareMathOperator{\Hom}{Hom}

\DeclareMathOperator{\Aut}{Aut}

\DeclareMathOperator{\id}{id}

\DeclareMathOperator{\Homeo}{Homeo}

\DeclareMathOperator{\SO}{SO}

\newcommand{\dmo}{\DeclareMathOperator}

\newcommand{\Om}{\Omega}

\newcommand{\wtil}{\widetilde}

\newcommand{\bb}[1]{\mathbb{#1}}

\dmo{\sgn}{sign}\dmo{\Span}{span}
\dmo{\we}{\wedge}
\dmo{\ind}{ind}\dmo{\Ind}{Ind}
\dmo{\bop}{\bigoplus}\dmo{\pic}{Pic}
\dmo{\vol}{Vol}\dmo{\gal}{Gal}\dmo{\perm}{Perm}
\dmo{\tor}{Tor}\dmo{\ext}{Ext}\dmo{\Ext}{Ext}
\dmo{\aut}{aut}

\dmo{\inn}{Inn}\dmo{\var}{Var}
\dmo{\ad}{ad}\dmo{\curl}{curl}
\dmo{\hy}{\bb H}\dmo{\Sl}{SL}
\dmo{\psl}{PSL}
\dmo{\iso}{iso}
\dmo{\conf}{Conf}
\dmo{\stab}{Stab}\dmo{\Jac}{Jac }
\dmo{\diam}{diam}\dmo{\fix}{Fixed}\dmo{\Fix}{Fix}
\dmo{\injR}{injRad}\dmo{\Ad}{Ad}
\dmo{\esv}{ess-vol}
\dmo{\nil}{Nil}\dmo{\sol}{Sol}
\dmo{\Div}{div}
\dmo{\SU}{SU}
\dmo{\rk}{rk}
\dmo{\rank}{rank}
\dmo{\psp}{PSp}\dmo{\psu}{PSU}
\dmo{\PU}{PU}\dmo{\pgl}{PGL}
\dmo{\Mod}{Mod}\dmo{\range}{Range}
\dmo{\eu}{eu}\dmo{\mi}{mi}
\dmo{\Log}{Log}\dmo{\supp}{supp}
\dmo{\maps}{Maps}\dmo{\Gr}{Gr}
\dmo{\Pin}{Pin}
\dmo{\Spin}{Spin}\dmo{\Str}{Str}
\dmo{\Sq}{Sq}\dmo{\Symp}{Symp}
\dmo{\pd}{PD}\dmo{\PD}{PD}\dmo{\sig}{Sig}
\dmo{\ev}{ev}\dmo{\St}{St}
\dmo{\Pt}{Pt}\dmo{\pt}{pt}

\dmo{\Pl}{PL}
\dmo{\String}{String}\dmo{\smear}{smear}

\dmo{\dev}{dev}
\dmo{\met}{Met}\dmo{\contact}{Contact}
\dmo{\teich}{Teich}\dmo{\Teich}{Teich}\dmo{\qi}{QI}
\dmo{\der}{Der}
\dmo{\cl}{Cliff}\dmo{\Cl}{Cl}
\dmo{\Pf}{Pf}
\dmo{\ch}{ch}\dmo{\diag}{diag}
\dmo{\grad}{grad}\dmo{\Char}{char}
\dmo{\spec}{Spec}\dmo{\Arg}{Arg}
\dmo{\gl}{GL}
\dmo{\sym}{Sym}\dmo{\Sym}{Sym}
\dmo{\com}{Comm}
\dmo{\Lk}{Lk}
\dmo{\CAT}{CAT}
\dmo{\Rep}{Rep}
\dmo{\Res}{Res}
\dmo{\Conf}{Conf}
\dmo{\PConf}{PConf}
\dmo{\Push}{Push}
\dmo{\Cont}{Cont}
\dmo{\sm}{\setminus}
\dmo{\vn}{\varnothing}
\dmo{\disk}{\mathbb D}
\dmo{\Trd}{Trd}\dmo{\Mat}{Mat}
\dmo{\Riem}{Riem}
\dmo{\Diffn}{\Diff_0}\dmo{\diff}{diff}
\dmo{\homeo}{Homeo}

\dmo{\Ham}{Ham}\dmo{\Met}{Met}
\dmo{\Ein}{Ein}\dmo{\CP}{\co P}
\dmo{\Per}{Per}\dmo{\Ric}{Ric}
\dmo{\Nrd}{Nrd}
\dmo{\Comp}{Comp}\dmo{\PSC}{PSC}
\dmo{\Cent}{Cent}\dmo{\Orb}{Orb}
\dmo{\aind}{a-ind}\dmo{\tind}{t-ind}
\dmo{\constant}{constant}
\dmo{\Td}{Td}
\dmo{\LMod}{LMod}
\dmo{\SMod}{SMod}
\dmo{\SDiff}{SDiff}
\dmo{\Br}{Br}
\dmo{\csch}{csch}
\dmo{\triv}{triv}
\dmo{\genus}{genus}
\dmo{\Homeq}{HomEq}
\dmo{\PP}{\mathbb{P}}
\dmo{\U}{U}
\dmo{\Gal}{Gal}
\dmo{\BDiff}{\wtil{\Diff}}
\dmo{\BAut}{\wtil{\Aut}}
\dmo{\Iso}{Iso}
\dmo{\Cone}{Cone}
\dmo{\codim}{codim}
\dmo{\II}{II}
\dmo{\I}{I}
\dmo{\InjRad}{InjRad}
\dmo{\Inn}{Inn}
\dmo{\sys}{sys}
\dmo{\Comm}{Comm}
\dmo{\PO}{PO}
\dmo{\vertex}{Vert}
\dmo{\POm}{P\Om}
\dmo{\ab}{ab}
\dmo{\PSO}{PSO}
\dmo{\CRS}{CRS}
\dmo{\Diffext}{Diffext}
\dmo{\Diffextad}{Diffextad}
\dmo{\Diffstand}{Diffstand}
\newcommand{\PDiff}{\mathrm{PDiff}}

\begin{document}
\title{The Euler class of infinite-type surface bundles}
\author{Mauricio Bustamante, Rita Jim\'enez Rolland, and  Israel Morales}

\begin{abstract}
We study the Euler class of smooth orientable infinite-type surface bundles with a section. For many such surfaces, we show that this cohomology class
is nontrivial, and that the behavior of its powers depends on the genus and the type of ends. As an application, we extend Morita's non-lifting theorem to many infinite-type surfaces, including surfaces of infinite genus.
\end{abstract}
\maketitle
\section{Introduction}
A fundamental problem in geometric topology is to classify all smooth fiber bundles with fiber a smooth manifold $S$. It is a well-known fact that isomorphism classes of such bundles over a fixed $CW$-complex $X$ are in bijective correspondence with homotopy classes $[X,B\Diff(S)]$ of continuous maps $X\to B\Diff(S)$, where $B\Diff(S)$ is the classifying space of the topological group $\Diff(S)$ of diffeomorphisms of $S$. In particular, the cohomology ring $H^*(B\Diff(S))$ is the ring of \textit{characteristic classes for smooth $S$-bundles}, which are among the most basic invariants one can assign to a fiber bundle.

The purpose of this paper is to study these characteristic classes for smooth fiber bundles whose fiber is a smooth oriented surface of \textit{infinite type}, namely a $2$-dimensional real smooth manifold whose fundamental group is not finitely generated.

We restrict our attention to oriented bundles equipped with a section, so as to incorporate more of the geometry of $S$. Equivalently, this amounts to studying the cohomology of $B\Diff_\ast(S)$, the classifying space of the group of orientation-preserving diffeomorphisms of $S$ that fix a point $\ast \in S$.
 
Given such a fixed point, we can assign to each diffeomorphism fixing $\ast$ its derivative at that point, which induces a map between classifying spaces
\beq
d:B\Diff_\ast(S)\to B\GL_2^+(\R),
\eeq
where $B\GL_2^+(\R)$ is the classifying space for real oriented vector bundles of rank 2.
The cohomology ring of $B\GL_2^+(\R)$ is a polynomial ring in the \emph{universal Euler class} $E\in H^2(B\GL_2^+(\R);\Z)$. 
Pulling back $E$ along $d$ yields
\beq
e:=d^*E\in H^2(B\Diff_\ast(S);\Z),
\eeq
which is a characteristic class for smooth $S$-bundles with a section. 
In this paper we study properties of the \emph{Euler class} $e$ for various infinite-type surface bundles.

To state our main results, it is useful to consider surfaces with more marked points. Let $\PDiff^k(S)$ denote the group of orientation-preserving diffeomorphisms of the surface $S$ which fix $k$ distinct \emph{marked points} $p_1,\ldots,p_k$ in $S$. Each point $p_i$ gives rise to an Euler class $e_i\in H^2(B\PDiff^k(S);\Z)$.
\begin{theoremalpha}\label{thm:pure-infinite}
For all oriented smooth surfaces $S$ of infinite genus, there is an injective ring homomorphism
\beq
\Q[e_1,\ldots,e_k]\hookrightarrow H^*(B\PDiff^k(S);\Q).
\eeq
\end{theoremalpha}
For infinite-type surfaces $S=Y_g$ of finite genus $g$, we obtain an analogous result which, in addition to the Euler classes, involves characteristic classes $\nu_i$ in degree $2i$ arising from the MMM-classes of $S_g$-bundles (with $S_g$ the compact orientable genus $g$ surface), obtained by fiberwise Freudenthal compactification of $Y_g$-smooth bundles.
\begin{theoremalpha}\label{thm:pure-finite}
For all oriented infinite-type smooth surfaces $Y_g$ of genus $g\geq 2$, and for all $k\geq 0$, there is a ring homomorphism
\beq
\Q[\nu_1,\ldots,\nu_{g-2},e_1,\ldots,e_k]\rightarrow H^*(B\PDiff^k(Y_g);\Q),
\eeq
which is injective up to degree $\lfloor \frac{2g-2}{3}\rfloor$.
\end{theoremalpha}
The definition of the classes $\nu_i$ and the proof of the theorem will be given in the Section \ref{sec:MMMPure}.

Specializing to other families of surfaces of infinite type, we obtain the following results.

\begin{theoremalpha}\label{thm:family1}
There is an uncountable family of non-diffeomorphic infinite-type surfaces of genus zero, including the Cantor tree surface, whose Euler class and all its powers have infinite order.
\end{theoremalpha}

\begin{theoremalpha}\label{thm:family2}
Let $X_g=S_g\setminus\mathcal{C}$, where $S_g$ is a closed oriented surface of genus $g\geq 1$ and $\mathcal{C}$ is a Cantor set embedded in $S_g$. The cohomology class
$e^n\in H^{2n}(B\Diff_\ast(X_g);\Z)$ is nontrivial for all $n\geq 1$. Moreover, 
for $n\geq g\geq 1$ the class $e^n$ is torsion of order at most the order of the $n$-th power of the Euler class in $H^{2n}(B\Diff_\ast(S_g);\Z)$.
\end{theoremalpha}

The Euler class can be used in questions about the existence of sections of the natural map
\beq
\Diff_\ast(S)\longrightarrow \Map_\ast(S),
\eeq 
where $\Map_{\ast}(S)$ denotes the \textit{mapping class group of $S$ relative to $\ast$}, which we define here as the group of path-components $\pi_0\Diff_\ast(S)$ of the topological group of orientation and base-point-preserving diffeomorphisms of $S$.

Morita noticed that the tautological class associated to the Euler class gives an obstruction to the existence of such sections. We obtain the following variant of \cite[Theorem 8.1]{morita}; see  Theorem \ref{thm:nonrealizability} for a more precise statement.
\begin{theoremalpha}\label{thm:Nielsen}
For all oriented smooth infinite-type surfaces of genus $g\in [13,\infty]$, and for an uncountable family of non-diffeomorphic genus zero surfaces, the map $\Diff_\ast(S)\longrightarrow \Map_\ast(S)$ admits no section.
\end{theoremalpha}
It would be interesting to know whether the non-lifting theorem also holds for homeomorphism groups. In this direction, Chen and He \cite[Theorem 1.3]{chen-he} show that this is the case for the Cantor tree surface.

We finish this introduction with some remarks.
\subsection{Finite-type surfaces}
Our results extend to infinite-type surfaces several theorems previously known for compact surfaces. In particular, when $S=S_g$ is compact of genus $g\geq 4$, Morita shows \cite[Theorem 7.5]{morita} that the Euler class and other tautological classes are algebraically independent up to degree $\lfloor g/3\rfloor$. By the improved stability range of Boldsen \cite[Theorems 1 \& 2]{boldsen}, this independence in fact holds up to degree $\lfloor (2g-2)/3\rfloor$. 

In addition, it is shown in work of the second-named author and Jekel \cite[Theorem A]{rita-jekel}, \cite[Theorem A]{rita-jekel2} that if $n\geq g\geq 1$, then $e^n$ generates a finite cyclic subgroup of order a multiple of $4g(2g+1)$, and the order of $e^g$ is divisible by $4g(2g+1)(2g-1)$. This can be compared to our Theorem \ref{thm:family2}.
\subsection{Cohomology of big mapping class groups}\label{sec:cohomlogy-bigmcg}
Our results can be interpreted as calculations of the cohomology of the mapping class group of infinite-type surfaces. This group is endowed with the quotient topology from the surjection $\Diff_\ast(S)\to\Map_\ast(S)$. With this topology, $\Map_*(S)$ is totally disconnected, and when $S$ has infinite type it is a Polish group. Let $\Map_*^\delta(S)$ denote the same group with the discrete topology. The map induced by the identity on classifying spaces
\beq
B\Map_*^\delta(S)\to B\Map_*(S)
\eeq
is a weak homotopy equivalence. Therefore there are canonical isomorphisms
\beq
H^*(B\Map_*(S);\Z)\cong H^*(B\Map_*^\delta(S);\Z)\cong H^*(\Map_*^\delta(S);\Z),
\eeq
where the last term is the group cohomology of $\Map_*^\delta(S)$.

Yagasaki \cite[Theorem 1.1]{yagasaki-homeo, yagasaki-diffeo} shows that the identity components of both $\Homeo(S)$ and $\Diff(S)$ are contractible, and the same holds for $\Homeo_\ast(S)$ and $\Diff_\ast(S)$. This implies that the identity components coincide with the respective path-components of the identity, giving weak homotopy equivalences
\beq
B\Diff_\ast(S)\to B\Map_*(S),\qquad B\Homeo_\ast(S)\to B\pi_0\Homeo_\ast(S).
\eeq
Smoothing theory further implies that the inclusion-induced map $B\Diff_\ast(S)\to B\Homeo_\ast(S)$ is a weak homotopy equivalence (see \cite[Lemma A.1]{palmer-wu-homology}). Therefore
\beq
H^*(B\Diff_\ast(S);\Z)\cong H^*((\pi_0\Homeo_\ast(S))^{\delta};\Z).
\eeq
In summary, our results can be viewed as computations in $$H^*((\pi_0\Homeo_\ast(S))^{\delta};\Z)\cong H^*(\Map^{\delta}_\ast(S);\Z),$$ which is often referred to in the literature as the cohomology of the \emph{big mapping class group}.
\subsection{Relation with other work} Let $\mathcal{C}$ be a Cantor set embedded in the $2$-sphere $S^2$. Calegari and Chen noticed in \cite[Appendix A]{CalegariChen2021} that $H^2(\Map_*^\delta({S}^2\setminus\mathcal{C});\mathbb{Z})$ contains an infinite cyclic subgroup, generated by the Euler class of the action on the simple circle, and Palmer and Wu \cite[Remark 1.1]{palmer-wu-homology} proved that $H^2(\Map_*^\delta({S}^2\setminus\mathcal{C});\mathbb{Z})\cong\mathbb{Z}$. 
In fact, they computed all the homology groups of the base-point-preserving mapping class group of any \emph{binary tree surface} {$\text{Gr}_{\mathfrak{B}}(\Sigma)$ associated to an oriented connected surface $\Sigma$ without boundary; see \cite[Definition 1.3]{palmer-wu-homology} and Section \ref{sec:genus zero}. This includes the surface $\text{Gr}_{\mathfrak{B}}(S^2)$ which is diffeomorphic to the \emph{Cantor tree surface} $S^2\setminus\mathcal{C}$. 
From their \cite[Corollary D]{palmer-wu-homology}  and the Universal Coefficient Theorem it follows that:
\beq
H^i(\Map_*^\delta(\text{Gr}_{\mathfrak{B}}(\Sigma));\mathbb{Z})=\begin{cases}
    \Z\text{\ if $i$ is even}\\
    0\text{\ if $i$ is odd}.\\
\end{cases}
\eeq
Our Theorem \ref{prop:binary tree surfaces} shows that, for any $n\geq 1$, the $n$-th power of the Euler class is an infinite order cohomology class in $H^{2n}(\Map_*^\delta(\text{Gr}_{\mathfrak{B}}(\Sigma));\mathbb{Z})$.

It is worth mentioning that in \cite[Section 8]{palmer-wu-JLMS} Palmer and Wu use the derivative to obtain homology classes dual to the powers of an Euler class (c.f. \cite[Section 3]{Bod-Till}) to show that if $S$ is a surface of infinite genus with a finite set of isolated ends (considered as `marked points'), then the inclusion of the subgroup of compactly supported mapping classes $\Map_c(S)$ in $\Map(S)$ induces a non-zero map on integral homology in every even degree. 
\subsection{Notation}
Throughout this paper, unless stated otherwise, surfaces are smooth, oriented and without boundary.
For such a surface $S$, we let $\Diff(S)$ denote the topological group (endowed with the compact-open $C^\infty$–topology) of orientation-preserving diffeomorphisms of $S$, and $\Homeo(S)$ denotes the topological group (with the compact-open topology) of orientation-preserving homeomorphisms of $S$. The subgroups consisting of diffeomorphisms or homeomorphisms fixing a marked point $\ast\in S$ will be denoted by $\Diff_\ast(S)$ and $\Homeo_\ast(S)$ respectively.

All cohomology groups are singular cohomology groups with integer coefficients, unless specified otherwise. 
\subsection{Organization of the paper}
In Section \ref{sec:background} we provide a minimal background on the classification of infinite-type surfaces and the Freudenthal compactification. Section \ref{sec:othereuler} presents several equivalent definitions of the Euler class.
Theorems \ref{thm:pure-infinite} and \ref{thm:pure-finite} are proved in Section \ref{sec:MMM}, where we also give an (indirect) definition of the MMM classes for finite-genus infinite-type surface bundles.
In Section \ref{sec:nontrivial} we prove Theorems \ref{thm:family1} and \ref{thm:family2}. Section \ref{sec:non-realization} illustrates an application of our results to the generalized Nielsen realization problem, Theorem \ref{thm:Nielsen}. Finally, in the Appendix \ref{sec:Nielsen-action} we explain how $\Map_\ast(S)$ acts faithfully on the circle and observe that such an action obstructs the splitting of the Birman exact sequence for many infinite-type surfaces.
\subsection{Acknowledgements}
We are grateful to N\'estor Colin, Rubén A. Hidalgo, and Eduardo Reyes for useful conversations. MB is supported by ANID Fondecyt Regular grant 1250727, and IM by ANID Fondecyt Postdoctoral grant 3240229. 
An important part of this project was achieved while RJR visited the Pontificia Universidad Católica de Chile. She thanks the university and MB for their hospitality.
\section{Background}\label{sec:background}
A surface is a 2-dimensional topological manifold without boundary.
Unless otherwise stated, we assume that all surfaces are connected, orientable, and smooth, that is they come equipped with a maximal $C^{\infty}$ atlas.  We say that a surface is of \textit{finite type} if its fundamental group is finitely generated, otherwise the surface is of \textit{infinite type}.
\subsection{Ends and classification of surfaces up to homeomorphism}
An \emph{end} of $S$ is an element of the inverse limit
\beq
\operatorname{Ends}(S)=\varprojlim_{K} \pi_0\bigl(S\setminus K\bigr),
\eeq
where the limit ranges over compact subsets $K\subset S$ ordered by inclusion.  The \emph{Freudenthal compactification} of $S$ is the union
\beq
\overline S \;=\; S \sqcup \operatorname{Ends}(S)
\eeq
equipped with the topology generated by sets of the form $U\;\sqcup\;\{e\in\operatorname{Ends}(S)\mid e<U\}$,
for $U\subset S$ open; here $e<U$ means that for some compact $K$ the component of $S\setminus K$ corresponding to $e$ is contained in $U$.  With this topology, $\operatorname{Ends}(S)$ is a compact, totally disconnected, separable space (in particular a closed subspace of the Cantor set).

An end $e\in\operatorname{Ends}(S)$ is called \emph{planar} if it has a neighborhood in $\overline{S}$ that embeds into the plane; otherwise $e$ is \emph{non-planar}.  We denote by $\operatorname{Ends}_{\mathrm{np}}(S)\subset\operatorname{Ends}(S)$ the (closed) subspace of non-planar ends.

The homeomorphism type of a connected, orientable surface without boundary is determined by its genus together with the pair of topological spaces $(\operatorname{Ends}(S),\operatorname{Ends}_{\mathrm{np}}(S))$. This classification theorem is due to Richards \cite{richards} and  Kerékjártó \cite{kerekjarto}. We only need the following weaker form of their theorem.

\begin{thm}\cite[Theorem 1]{richards}\label{thm:richards}
Let $S_1,S_2$ be connected, orientable surfaces without boundary, of genera $g_1,g_2\in\mathbb{N}\cup\{\infty\}$.  Then $S_1\cong S_2$ (homeomorphic) if and only if $g_1=g_2$ and there is a homeomorphism of pairs
\beq
\bigl(\operatorname{Ends}(S_1),\operatorname{Ends}_{\mathrm{np}}(S_1)\bigr)
\cong
\bigl(\operatorname{Ends}(S_2),\operatorname{Ends}_{\mathrm{np}}(S_2)\bigr).
\eeq
\end{thm}
\subsection{Classification of surfaces up to diffeomorphism}
Two smooth surfaces are homeomorphic if and only if they are diffeomorphic. Therefore, the invariants that determine the homeomorphism type of an orientable surface suffice to determine their diffeomorphism type. This is a well-known fact which can be proved by assembling the following chain of deep results in manifold topology.

Firstly, T. Rad\'o shows that every topological surface admits a PL-structure \cite[Theorem 8.3]{moise}. Moise (Theorem 8.4) loc. cit. shows the \textit{Hauptvermutung} for surfaces, namely two homeomorphic surfaces are $PL$-equivalent.

Secondly, it follows from \cite[Theorem 6.5]{munkres} (and the fact that the first singular homology group of an orientable surface is free abelian) that $PL$-equivalent surfaces are diffeomorphic. Note that Munkres' requires diffeomorphisms to be only of class $C^1$, however $C^1$-diffeomorphic manifolds are $C^{\infty}$- diffeomorphic (see \cite[Ch. 2. Theorem 2.9]{hirsch}). 
\subsection{Extension of diffeomorphisms to the Freudenthal compactification}\label{sec:extension}
As before, consider $\overline S$, the Freudenthal compactification of $S$, and let $\widehat S$ denote the maximal subspace of $\overline{S}$ that is a surface (we think of $\widehat S$ as the result of ``filling in'' all the planar ends of $S$). For any $f\in\Homeo(S)$ define $\overline{f}\in\Homeo(\overline S)$ by
\beq
\overline{f}(x)
\begin{cases}
f(x) & \text{if}\ \ x\in S\\
f_\ast(x)& \text{if}\ \ x\in \mathrm{End}(S)
\end{cases}
\eeq
where $f_\ast(x)$ is the map on the inverse limit induced by the component maps 
$\pi_0\bigl(S\setminus K\bigr)\to \pi_0\bigl(S\setminus f(K)\bigr)$,
for each compact $K\subset S$. The assignment $f\mapsto\overline{f}$ defines a group homomorphism. Since any homeomorphism of $\overline{S}$ preserves $\widehat S$, there are well-defined continuous homomorphisms 
\beq
\Homeo(S)\rightarrow \Homeo(\overline{S})\rightarrow \Homeo(\widehat S).
\eeq
Precomposing with the inclusion $\Diff(S)\hookrightarrow\Homeo(S)$  yields the continuous homomorphism 
\beq
\kappa:\Diff(S)\to\Homeo(\widehat{S}).
\eeq
The surface $S$ lies in $\widehat S$ as an open dense subset, and so we can choose a marked point $\ast\in S\subset \overline{S}$. By a slight abuse of notation we also use $\kappa:\Diff_{\ast}(S)\to\Homeo_{\ast}(\widehat{S})$
\begin{rmk}
Even in the case when $S$ has only planar ends, so that $\widehat{S}=\overline{S}$ has a smooth structure, the image of the map $\kappa$ may not be contained in $\Diff(\widehat{S})$. As an example, consider the surface $X_g$ obtained by removing a Cantor set $\calC$ from the the closed orientable surface $S_g$ of genus $g$. Then $\widehat{X_g}$ is a surface homeomorphic and hence diffeomorphic to $S_g$. If the extension of every diffeomorphim of $X_g$ to its Freudenthal compactification $\widehat{X_g}\cong S_g$ were also a diffeomorphism, we would have the following commutative diagram of topological groups
\beq
\xymatrix{
\Diff(X_g)\ar[r]^-{\kappa}\ar[d]& \Diff(S_g,\calC)\ar[d]\\
\Homeo(X_g)\ar[r]^-{\kappa}& \Homeo(S_g,\calC)
}
\eeq
where the groups on the right are diffeomorphisms and homeomorphism of $S_g$ that preserve $\calC$ setwise.
The lower horizontal map is an isomorphism (c.f. \cite[Corollary B.3]{palmer-wu-Documenta}) and by the discussion in Section \ref{sec:cohomlogy-bigmcg} the left vertical map is a weak homotopy equivalence. This would yield an injective map on mapping class groups
\beq
\kappa_\ast:\Map(X_g)\hookrightarrow\Map(S_g,\calC),
\eeq
but this is impossible because the domain is uncountable whereas the codomain is countable, as a consequence of  \cite[Theorems 1 $\&$ 3]{funar-neretin}.
\end{rmk}
\section{Alternative constructions of the Euler class}\label{sec:othereuler}
Recall that we have defined the Euler class $e$ as the pullback of the universal Euler class $E\in H^2(B\GL_2^+(\R))$ along the map $d:B\Diff_\ast(S)\to B\GL_2^+(\R)$ induced by the derivative at the marked point $\ast\in S$. It will be convenient to view $e$ as arising from other constructions.
\subsection{The Euler class via the vertical tangent bundle}\label{sec:vertical-euler}
The universal smooth fibre bundle with fibre $S$ may be described in terms of classifying spaces as
\beq
S\to B\Diff_\ast(S)\xto{\pi} B\Diff(S).
\eeq
The \textit{vertical tangent bundle} of $\pi$ is the real vector bundle of rank 2 over $B\Diff_\ast(S)$ consisting of all vectors tangent to the fibers of $\pi$. More precisely, the universal principal $\Diff(S)$-bundle $\Diff(S)\to E\Diff(S)\to B\Diff(S)$ coincides with associated principal $\Diff(S)$-bundle of $\pi$. The group $\Diff(S)$ acts on the tangent bundle $TS$ through the derivative and the vertical tangent bundle is, by definition, the rank $2$ vector bundle
\begin{equation}\label{eq:vertical}
E\Diff(S)\times_{\Diff(S)}TS\to E\Diff(S)\times_{\Diff(S)} S\simeq B\Diff_\ast(S),
\end{equation}
which is classified by its Euler class in $H^2(B\Diff_\ast(S);\Z)$. This Euler class is exactly $e=d^\ast E$. In fact, note that the vector bundle \eqref{eq:vertical} is isomorphic to the vector bundle 
\beq
E\Diff_\ast(S)\times_{\Diff_\ast(S)}T_\ast S\to E\Diff_\ast(S)\times_{\Diff_\ast(S)}\ast\simeq B\Diff_\ast(S)
\eeq
which is nothing but the pullback of the canonical $2$-plane bundle over $B\GL_2^+(\R)$ along the derivative $d:B\Diff_\ast(S)\to B\GL_2^+(\R)$. 

This perspective on the Euler class will be used in Sections \ref{sec:MMM} and \ref{sec:non-realization}.
\subsection{The Euler class via the action on the circle}\label{sec:euler-action}

There is a geometrically defined action of $\Map_\ast(S)$ on the circle $S^1$, often called \emph{the Nielsen action}, which is roughly given by first lifting a diffeomorphism $f:S\to S$ to a diffeomorphism $\widetilde{f}:\Hyp\to\Hyp$ of the universal cover, and then extending it to a homeomorphism $\partial\widetilde{f}:\partial_{\infty}\Hyp\to\partial_{\infty}\Hyp$ of the circle at infinity.  This action was first defined by Nielsen when $S$ is a closed orientable hyperbolic surface (c.f. \cite[Sections 5.5.2 $\&$ 8.2.5]{FaMa12}), and it has recently been studied by Tappu \cite[Section 4]{Tappu2023} for orientable Nielsen-convex surfaces of infinite type;  see also \cite[Proposition 5.3]{Thurston}. We recall Tappu's generalization of the Nielsen action in Appendix \ref{sec:Nielsen-action}, and observe there that the action on the circle he defines is faithful, continuous, and by orientation preserving homeomorphisms. The action obtained by precomposing this action with the  projection $\Diff_\ast(S)\to\Map_\ast(S)$ will be denoted from now on by
\begin{equation}\label{Tappu}
\eta:\Diff_\ast(S)\to\Homeo^+(S^1).
\end{equation}

There is a deformation retraction of $\Homeo^+(S^1)$ onto the subgroup $\SO(2)$ of rotations \cite[Proposition 4.2]{ghys}, so we can identify the cohomology rings
\beq
H^*(B\Homeo^+(S^1))\cong H^*(B\SO(2))\cong\Z[E].
\eeq
Therefore, if $S$ is an orientable Nielsen convex surface of infinite type with a marked point, we can pullback $E$ along the map $B\eta$ on classifying spaces induced by $\eta$ to obtain the class
\beq
(B\eta)^*E\in H^2(B\Diff_\ast(S)).
\eeq

\begin{prop}\label{prop:euler-classes-coincide}
For all oriented Nielsen-convex hyperbolic surfaces $S$, the classes $(B\eta)^*E$ and $e=d^*E$ are equal.
\end{prop}
The proposition is due to Morita  \cite[Proposition 4-1]{Morita-bounded} in the case when $S$ is compact and working directly with the Euler class of the vertical tangent bundle defined above (see also  	\cite[Proposition 3.1]{hamenstaedt}). Bestvina--Church--Souto \cite[Lemma 3.2]{bestvina-church-souto} give a different argument to prove that the Euler classes agree when restricted to the ``push'' subgroup $\pi_1(S,\ast)$ of the mapping class group of $S$. Our proof mimics theirs.
\begin{proof}
Let $\overline{\Hyp}=\Hyp\cup\partial_{\infty}\Hyp$ be the geodesic compactification of the hyperbolic plane $\Hyp$. Fix a $y\in\Hyp\subset\overline{\Hyp}$ in the preimage of $\ast\in S$ under the universal covering projection $\Hyp\to S$. Identify $T_\ast S$ with $T_y\Hyp$ and view the tangent circle $S_y\Hyp$ at $y\in\Hyp$ as equivalence classes of rays in $T_y\Hyp$ emanating from the origin. Points in $S_y\overline{\Hyp}$ will be denoted by $[v]$, where $v$ is a unit tangent vector at $y$. Consider the space
Let $\overline{\Hyp}=\Hyp\cup\partial_{\infty}\Hyp$ be the geodesic compactification of the hyperbolic plane $\Hyp$. Fix a $y\in\Hyp\subset\overline{\Hyp}$ in the preimage of $\ast\in S$ under the universal covering projection $\Hyp\to S$. Identify $T_\ast S$ with $T_y\Hyp$ and view the tangent circle $S_y\Hyp$ at $y\in\Hyp$ as equivalence classes of rays in $T_y\Hyp$ emanating from the origin. Points in $S_y\Hyp$ will be denoted by $[v]$, where $v$ is a unit tangent vector at $y$. Consider the space
\beq
\mathrm{bl}_y(\overline{\Hyp}):=\left\{(\xi,[v])\in\overline{\Hyp}\times S_y\Hyp\ |\ \xi=\exp_y(tv)\ \ \text{for some}\ \ t\in [0,\infty]\right\}.
\eeq 
Note that $\mathrm{bl}_y(\overline{\Hyp})$ is homeomorphic to a closed annulus where one boundary component is $\partial_\infty\Hyp$ and the other is $S_y\Hyp$.

A diffeomorphism $f:S\to S$ fixing $\ast$ lifts uniquely to a diffeomorphism $\widetilde{f}:\Hyp\to\Hyp$ fixing $y$. This lift extends to a homeomorphism $\widehat{f}:\mathrm{bl}_y(\overline{\Hyp})\to\mathrm{bl}_y(\overline{\Hyp})$ as follows:
\beq
\widehat{f}(\xi,[v])=(\overline{f}(\xi),[d\widetilde{f}_y(v)]),
\eeq
where $\overline{f}(\xi)=\widetilde{f}(\xi)$ if $\xi\in\Hyp$ and $\overline{f}(\xi)=\partial\widetilde{f}(\xi)$ if $\xi\in\partial_{\infty}\Hyp$. Here, the map $\partial\widetilde{f}$ is the extension of $\widetilde{f}$ to the geodesic boundary $\partial_{\infty}\Hyp$ (see Appendix \ref{sec:Nielsen-action}).

If we identify $\mathrm{bl}_y(\overline{\Hyp})$ with the cylinder $S^1\times [0,1]$ in a way that $S^1\times\{0\}$ corresponds to $S_y\Hyp$, then the assignment $f\mapsto \widehat{f}$
yields a homomorphism
\beq
\Phi:\Diff_\ast(S)\to \Homeo(S^1\times [0,1]).
\eeq
The $S^1\times [0,1]$-bundle
\beq
\mathcal{E}:=E\Diff_\ast(S)\times_{\Diff_\ast(S)}\left( S^1\times [0,1]\right)\to B\Diff_\ast(S)
\eeq
where the action of $\Diff_\ast(S)$ on $S^1\times [0,1]$ is via $\Phi$, contains the subbundles $$\mathcal{E}_i:=E\Diff_\ast(S)\times_{\Diff_\ast(S)}\left( S^1\times \{i\}\right),$$ $i=0,1$ which are the circle bundles induced by the derivative and by $\eta$ respectively. Since the inclusions $\mathcal{E}_i\hookrightarrow\mathcal{E}$ are fiber-preserving homotopy equivalences, the bundles $\mathcal{E}_0$ and $\mathcal{E}_1$ are fiber homotopy equivalent. Therefore $d^*E=(B\eta)^*E$.
\end{proof}
This point of view on the Euler class will be combined with Proposition \ref{prop:euler-freudenthal} below to analyze the case of bundles of surfaces of finite genus in Section \ref{sec:finite genus}.
\subsection{The Euler class and the planar ends}\label{sec:Euler-compact}
Let $S$ be an orientable Nielsen-convex surface and $\widehat{S}$ be the surface obtained by ``closing up'' all the planar ends of $S$ (c.f. Section \ref{sec:extension}). Take a marked point $\ast\in S\subset \widehat{S}$ and  let $\kappa:\Diff_\ast(S)\to\Homeo_\ast(\widehat{S})$ defined as in Section \ref{sec:extension}. As $\widehat{S}$ is also a Nielsen-convex surface, we can form the following (noncommutative) diagram
\beq
\xymatrix{
\Diff_\ast(S)\ar[dr]_{\eta_S}\ar[r]^-{\kappa}&\Homeo_\ast(\widehat{S})\ar[d]^{\eta_{\widehat{S}}}\\
&\Homeo^+(S^1)
}
\eeq
where $\eta_S$ and $\eta_{\widehat{S}}$ are the actions on the circle introduced in \eqref{Tappu} and defined in Appendix \ref{sec:Nielsen-action}.
\begin{prop}\label{prop:euler-freudenthal}
Let $d_S:B\Diff_\ast(S)\to B\GL_2^+(\R)$ be the map induced by the derivative at $\ast\in S$. Then
$(B\kappa)^*
(B\eta_{\widehat{S}})^*E
=(B\eta_S)^*E=d_S^*E=e$.
\end{prop}
\begin{proof}
The second equality is Proposition \ref{prop:euler-classes-coincide} applied to the surface $S$. For the first equality, we argue as in its proof, but this time we use the universal covering projection $\Hyp\to\widehat{S}$ to construct, for each $f\in\Diff_\ast(S)$, a self-homeomorphism $\check{f}:\mathrm{bl}_{y}(\overline{\Hyp})\to\mathrm{bl}_{y}(\overline{\Hyp})$ of the blowup, defined by
\beq
\check{f}(\xi,[v])=(\overline{\kappa(f)}(\xi),[d\widetilde{f}_y(v)]).
\eeq
Here we use that $\kappa(f)$ agrees with $f$ on an open neighborhood of $\ast\in S\subset\widehat{S}$, which implies $d\widetilde{f}_y=d\widetilde{\kappa(f)}_y$.

From this point, the same argument  as above produces a fiber homotopy equivalence between the circle bundles induced by $d_S$ and $B\kappa\circ
B\eta_{\widehat{S}}$, and the result follows.
\end{proof}
In this sense, the Euler class does not distinguish between bundles of surfaces with planar ends and their fiberwise compactifications. This is consistent with \cite[Remark B.4]{palmer-wu-Documenta}. We will exploit this fact in Section \ref{sec:finite genus} to get upper bounds on the torsion of certain powers of the Euler class of infinite type surfaces with finite genus.
\section{Proofs of Theorems \ref{thm:pure-infinite} and \ref{thm:pure-finite}}\label{sec:MMM}
Throughout this section it is useful to think of the Euler class as defined in Section \ref{sec:vertical-euler}. We will prove Theorems \ref{thm:pure-infinite} and \ref{thm:pure-finite}. For convenience we state them here as Theorem \ref{prop:marked-points} and \ref{prop:MMM} respectively.
\subsection{Infinite-genus surfaces}\label{sec:pure}   First we focus on studying the behavior of the Euler classes for oriented smooth surfaces of infinite genus.
\begin{thm}\label{prop:marked-points}
Let $S$ be a surface of infinite genus. Then, for all $k\geq 1$, there is an injective ring homomorphism\footnote{Our classes $e_i$ correspond to the classes $\sigma_i$ in Morita's article.}
\beq
\Q[e_1,\ldots,e_k]\hookrightarrow H^\ast(B\PDiff^k(S);\Q).
\eeq
\end{thm}
\begin{proof}
We follow \cite[Theorem 7.5]{morita}.
Let $S_{g_1,1},\ldots, S_{g_k,1}$ surfaces of genus $g_j$ with one boundary component.  Consider an embedding of the disjoint union of the $S_{g_i,1}$ in $S$ 
\beq
S_{g_1,1}\sqcup\cdots\sqcup S_{g_k,1} \hookrightarrow S
\eeq
such that $p_i\in S_{g_i,1}$. This embedding induces a map
\beq
\iota:B\Diff_{p_1}(S_{g_1,1})\times\cdots\times B\Diff_{p_k}(S_{g_k,1})\to B\PDiff^k(S).
\eeq
Let $\pi_j:B\Diff_{p_1}(S_{g_1,1})\times\cdots\times B\Diff_{p_k}(S_{g_k,1})\to B\Diff_{p_j}(S_{g_j,1})$  be the projection onto the $j$-th factor, and denote by $\varepsilon_j$ the Euler classes in $H^2(B\Diff_{p_j}(S_{g_j,1}))$. Then
\beq
\iota^*(e_j)=\pi_j^*(\varepsilon_j).
\eeq
Now let $p(e_1,\ldots,e_k)=\sum a_{i_1,\cdots,i_k}e_1^{i_1}\cdots e_k^{i_k}$ be a polynomial in the Euler classes. Choose the genus $g_j$ large enough so that the monomial in $e_j$ is nontrivial when restricted to $H^*(B\Diff_{p_j}(S_{g_j,1});\mathbb{Q})$. If $p(e_1,\ldots,e_k)=0$ then
\beq
p(\iota^*e_1,\ldots,\iota^*e_k)=p(\pi_1^*\varepsilon_1,\ldots,\pi_k^*\varepsilon_k)=0.
\eeq
By K\"unneth's formula, the class $\pi_j^*\varepsilon_j$ coincides with the restriction of $e_j$ to $B\Diff_{p_j}(S_{g_j,1})$, so the polynomial vanishes only when it is identically zero.
\end{proof}
\begin{rmk}
It is worth mentioning that for surfaces $S$ of infinite genus with nonplanar ends, Palmer--Wu \cite[Theorem 8.7]{palmer-wu-JLMS} show that the duals of the Euler classes are linearly independent in the homology of the subgroup of $\Map(S\setminus\{p_1,\ldots,p_k\})$ consisting of mapping classes with compactly supported representatives.
\end{rmk}
\subsection{MMM classes and finite-genus surfaces}\label{sec:MMMPure}
Recall that Mumford \cite{mumford} , Miller \cite{miller} and Morita \cite{morita} defined the \emph{MMM classes} $\mu_i$ in $H^{2i}(B\Diff(S_g); \mathbb{Q})$  for a closed orientable surface $S_g$ of genus $g\geq 2$, by integration  of the $(i + 1)$-th power of the tangential Euler class in the universal smooth $S_g$-bundle (see Section \ref{sec:vertical-euler}). 
Integration along the fibers is a trivial map when $S$ is a surface of infinite type, and so the MMM classes do not have an immediate analog in this case. Nevertheless, for orientable  surfaces of finite genus, we can pullback these $\mu_i$-classes along the map induced by the Freudenthal compactification. More precisely, as before let $Y_g$ denote an  oriented surface  of genus $g\geq 2$. For $i\geq 1$, we define cohomology classes by 
\beq 
\nu_i:=(B\kappa)^*\mu_i\in H^{2i}(B\Diff(Y_g);\Q),
\eeq
where $B\kappa:B\Diff(Y_g)\to B\Homeo(S_g)$ is the map induced by the Freudenthal compactification defined in Section \ref{sec:extension}.
\begin{thm}\label{prop:MMM}Let $Y_g$ an oriented smooth infinite-type surface of finite genus $g\geq 2$ and $k\geq 0$.
There is an injective ring homomorphism
\beq
\Q[\nu_1,\ldots,\nu_{g-2},e_1,\ldots,e_k]\rightarrow H^\ast(B\PDiff^k(Y_g);\Q)
\eeq
up to degree $\lfloor\frac{2g-2}{3}\rfloor$.   
\end{thm}

\begin{proof}
Since $Y_g$ has finite genus $g$, all of its ends are planar and it is diffeomorphic
to $S_g\setminus E$, where $E$ is the image of an embedding of $\mathrm{Ends}(Y_g)$ in $S_g$. Consider an embedded disc $D^2 \subset S_g$ containing $E$ in its interior, and take the marked points (if any) $p_1,\dots, p_k$ in $S_g\setminus \mathrm{int} (D^2)$.  
Let $\PDiff_{\partial}^k(S_{g,1})$ denote the group of diffeomorphisms of $S_g$ that fix the embedded disk $D^2$ and the marked points pointwise.  There is the following commutative diagram
\beq
\xymatrix{
& B\PDiff_{\partial}^k(S_{g,1})\ar[rr]^{\epsilon}\ar[rd]_{c}& &B\PDiff^k(Y_g) \ar[ld]^{B\kappa}\\
& &  B\mathrm{PHomeo}^k(S_g)
}
\eeq
where the horizontal map $\epsilon$ is induced by the extension by the identity homomorphism $\PDiff^k_{\partial}(S_{g,1})\to\PDiff^k(Y_g)$, and the map $c$ is induced by the inclusions (or capping homomorphisms) $\PDiff^k_\partial(S_{g,1})\hookrightarrow \PDiff^k(S_g)\hookrightarrow \mathrm{PHomeo}^k(S_g)$. By \cite[Theorem 7.5]{morita}, with the improved stability range from \cite[Theorems 1 $\&$ 2]{boldsen},  there is an injective ring homomorphism
\beq
\Q[\mu_1,\ldots,\mu_{g-2},e_1,\ldots,e_k]\to H^\ast(B\PDiff^k(S_{g});\Q),
\eeq
up to degree $\lfloor \frac{2g-2}{3}\rfloor$, and the map in cohomology induced by $c$ restricts to a monomorphism on $\Q[\mu_1,\ldots,\mu_{g-2},e_1,\ldots,e_k]$ in the same range of degrees. The claim then follows from the commutativity of the diagram.
\end{proof}
\begin{rmk} In contrast with Theorem \ref{prop:MMM},  the dual MMM classes vanish in the rational homology of the Loch Ness monster surface \cite[Theorem A]{palmer-wu-JLMS}. 
\end{rmk}
\section{Proof of Theorems \ref{thm:family1} and \ref{thm:family2}}\label{sec:nontrivial}  
The following is a ``fiber-bundle version'' of \cite[Proposition 3.1]{rita-jekel}. It gives us a criterion to prove non-triviality of the Euler class and its powers to obtain Theorems \ref{thm:family1} and \ref{thm:family2} below. 
\begin{lem}\label{lem:nontrivial-euler}
Let $S$ be an oriented smooth surface  with one marked point. If $\Diff_\ast(S)$ contains a finite cyclic subgroup $C_m$ of order $m>1$  
then  $e^n\neq 0$ in $H^{2n}(B\Diff_\ast(S))$ for all $n\geq 1$. Moreover, if  $e^n$ has finite order then  it is divisible by $m$.
\end{lem}
\begin{proof}
Let $f:S\to S$ be the generator of $C_m<\Diff_\ast(S)$. By \cite[Theorem 2.8]{epstein-periodic} there exists a complete hyperbolic metric on $S$ such that $f$ is an isometry. In particular, at the neighborhood of the fixed point $\ast\in X$, the derivative $df_*$ is a rotation by $2\pi/m$. 
Therefore the pullback of the universal oriented plane bundle over $B\GL_2^+(\R)$ along $d:BC_m\to B\GL_2^+(\R)$ is isomorphic to the bundle $EC_m\times_{C_m}\R^2\to BC_m$, where $C_m$ acts on $\R^2\cong T_\ast S$ via the derivative of $f$ at $\ast$. The Euler class of this bundle is the pullback of the universal Euler class $E\in H^2(B\GL_2^+(\R))$ under the derivative map. Thus it suffices to show that $d^*E$ is the generator of the group $H^2(BC_m)\cong C_m$, because then one uses that there is a ring isomorphism $H^*(BC_m)\cong\Z[x]/\langle mx\rangle$.

Since the action of $C_m$ by rotation is equivalent to complex multiplication by $\omega=e^{2\pi i/m}$ when $\R^2$ is identified with $\C$, this vector bundle becomes a complex line bundle $L\to BC_m$ whose first Chern class $c_1(L)\in H^2(BC_m)\cong H^2(C_m)$ corresponds, under the isomorphism $\Hom(C_m,U(1))\cong H^2(C_m)$, to the representation $f_m:C_m\to U(1)$ given by multiplication by $\omega$. Recall that the aforementioned isomorphism is given by sending a representation $f:C_m\to U(1)$ to the pullback of the central extension $0\to\Z\to\R\xto{\exp} U(1)\to 0$ along $f$. The generator corresponds to the representation whose pullback is the extension
\beq
0\to\Z\xto{\cdot m}\Z\to C_m\to 0,
\eeq
that is to $f_m$. Therefore $c_1(L)=d^*E$ is the generator of $H^2(BC_m)$.
\end{proof}
\begin{cor}\label{coro:criterion-allorders}
Let $\{m_i\}$ be an increasing sequence of natural numbers. If $\Diff_\ast(S)$ contains finite cyclic subgroups of order $m_i$ for all $i=1,2,\ldots$, then $e^n$ has infinite order in $H^{2n}(B\Diff_\ast(S))$.
\end{cor}
\begin{proof}
If the order of $e^n$ were finite, then by Lemma \ref{lem:nontrivial-euler}, the infinitely many positive integers $m_i$ would divide its order, which is impossible.
\end{proof}
\subsection{Surfaces of genus zero}\label{sec:genus zero} In this section we prove Theorem \ref{thm:family1}, which is stated as Theorem \ref{prop:binary tree surfaces}. We consider a family of genus zero surfaces  (c.f. \cite[Section 1.2]{palmer-wu-homology})  for which all powers of their Euler class have infinite order. 

Let $\mathfrak{B}$ denote the infinite binary tree, and consider a connected smooth surface $\Sigma$ without boundary. For each vertex $v$ of $\mathfrak{B}$, let $\Sigma_v$  be the surface obtained by removing $\deg(v)$ disjoint closed disks of $\Sigma$, and equipped with a labelling of its boundary circles by the edges of $\mathfrak{B}$ incident to $v$. The \emph{binary tree surface associated to $\Sigma$} is the surface
$$\mathrm{Gr}_{\mathfrak{B}}(\Sigma):=\bigsqcup_v \Sigma_v\big/\sim$$
 where if $e$ is an edge of $\mathfrak{B}$ incident to $v$ and $w$, we glue the boundary circle of $\Sigma_v$ labelled by $e$ to the boundary circle of  $\Sigma_w$ labelled by $e$ by an orientation-preserving diffeomorphism.
 
Notice that the  surface $\mathrm{Gr}_{\mathfrak{B}}(\Sigma)$ is a surface without boundary that has either genus zero or infinity.  In particular, $\mathrm{Gr}_{\mathfrak{B}}(S^2)$ is  the \emph{Cantor tree surface} $S^2\setminus \mathcal{C}$  and $\mathrm{Gr}_{\mathfrak{B}}(T^2)$  is the \emph{Cantor blooming tree surface}. As pointed out in \cite[Remark 1.7]{palmer-wu-JLMS} there are uncountably many non-homeomorphic (and hence non-diffeomorphic) surfaces  $\mathrm{Gr}_{\mathfrak{B}}(\Sigma)$ of genus zero.
\begin{thm}\label{prop:binary tree surfaces}
Let $\Sigma$ be an oriented smooth surface, and  $\mathrm{Gr}_{\mathfrak{B}}(\Sigma)$ its associated binary tree surface. Then the Euler class $e$ and all its powers have infinite order in the cohomology of $B\Diff_\ast(\mathrm{Gr}_{\mathfrak{B}}(\Sigma))$.
\end{thm}
\begin{proof} 
If $\Sigma$ has positive genus, the result is covered by Theorem \ref{thm:pure-infinite}. Assume $\Sigma$ has genus zero, and let $n$ be a positive integer. Label the vertex set of the infinite binary tree $\mathfrak{B}$ with $V=\{\phi\}\sqcup\{\text{finite sequences of $0$ and $1$'s}\}$. 
For each vertex $v\in V$ (including $\phi$) of length $\leq n$, let $\Sigma_v$  the surface obtained by removing $\deg(v)$ disjoint closed disks of the sphere $S^2$, and equipped with a labelling of its boundary circles by the edges of $\mathfrak{B}$ incident to $v$. 
And for each vertex $v\in V$ of length $> n$, let $\Sigma_v$   the surface obtained by removing $\deg(v)$ disjoint closed disks of $\Sigma$, and equipped with a labelling of its boundary circles by the edges of $\mathfrak{B}$ incident to $v$. Let  
$$\mathrm{Gr}^n_{\mathfrak{B}}(\Sigma):=\bigsqcup_{v\in V} \Sigma_v\big/ \sim$$
 where if $e$ is an edge of $\mathfrak{B}$ incident to $v$ and $w$, we glue the boundary circle of $\Sigma_v$ labelled by $e$ to the boundary circle of  $\Sigma_w$ labelled by $e$ by an orientation-preserving diffeomorphism. From the construction, one sees that there is an order $2^n$ diffeomorphism of $\mathrm{Gr}^n_{\mathfrak{B}}(\Sigma)$ with a fixed point. By the classification of surfaces $\mathrm{Gr}^n_{\mathfrak{B}}(\Sigma)$ is diffeomorphic to $\mathrm{Gr}_{\mathfrak{B}}(\Sigma)$. Corollary \ref{coro:criterion-allorders} implies that the Euler class and all its powers are nontrivial. 
\end{proof}
\subsection{Surfaces of finite positive genus}\label{sec:finite genus} Now let $E\subset\calC$ be a closed subset of the Cantor set which is realized as the space of ends of a genus $g\geq 1$ orientable surface $Y_g$.
\begin{prop}\label{prop:Yg}
Let $\ast\in Y_g$ be a marked point and $n$ a positive integer. If $n\geq g\geq 1$ and $e^n\neq 0$, then
$e^n$ is a torsion class whose order is  bounded above by  a multiple of $4g(2g+1)$.
\end{prop}
\begin{proof}
Recall that by Proposition \ref{prop:euler-freudenthal} the Euler class $e=d^*E$ coincides with the pullback of the Euler class in $H^2(B\Diff_\ast(S_g))$ along the map induced by the Freudenthal compactification $\kappa:\Diff_\ast(Y_g)\to\Homeo_\ast(S_g)$.  Then the result follows from \cite[Theorem A]{rita-jekel}.
\end{proof}
Except for the following special case, we do not know whether the $e^n\neq 0$ when $n\geq g\geq 4$ for a general surface $Y_g$. We obtain  Theorem \ref{thm:family2} restated as follows. 
\begin{thm}\label{prop:nontrivialXg}
Let $X_g=S_g\setminus\mathcal{C}$, where $S_g$ is a closed oriented surface of genus $g\geq 1$ and $\mathcal{C}$ is a Cantor set embedded in $S_g$. The Euler class $e\in H^2(B\Diff_\ast(X_g))$ and all its powers are nontrivial. Furthermore, for $g\geq 3n+1$ the class $e^n$ has infinite order, and for $n\geq g$ the power $e^n$ has finite order bounded above by a multiple of $4g(2g+1)$.
\end{thm}
\begin{proof}
By Lemma \ref{lem:nontrivial-euler}, to show the first part we only need to  exhibit one element of finite order in $\Diff_\ast(X_g)$. 
Note that  the Cantor set $\calC$ is homeomorphic to the $2$-fold disjoint union $\calC\sqcup
\calC$.  Now let the cyclic group $C_2$ of order $2$ act on the closed orientable surface of genus $g\geq 1$ by the hyperelliptic involution. By removing symmetrically a Cantor set $\calC\sqcup\calC$ in neighborhoods of a free orbit, we obtain an action of $C_2$ on $X_g$ with a fixed point (in fact, with $2g+2$ fixed points). This shows that the Euler class and all its powers are nontrivial. Now apply Proposition \ref{prop:Yg}.
\end{proof}
\begin{rmk}
A consequence of Corollary \ref{coro:criterion-allorders} and Theorem \ref{prop:nontrivialXg} is that the order of the torsion elements in $\Diff_{\ast}(X_g)$ must be uniformly bounded. In fact, it can be shown that any torsion element of $\Map_{\ast}(Y_g)$ has order at most $4g+2$. 
\end{rmk}
For the surfaces $X_g$ with $g\geq 3$, there is the following alternative way to prove that the Euler class has infinite order.
\begin{prop}\label{prop:infinite-euler}
If $g\geq 3$, then the derivative map induces an injective homomorphism
\beq
d^\ast:\Z\cong H^2(B\GL_2^+(\R))\to H^2(B\Diff_\ast(X_g)).
\eeq
Consequently $e=d^*E\in H^2(B\Diff_\ast(X_g))$ has infinite order.
\end{prop}
\begin{proof}
Let $S^1_g=S_g\setminus\ast$. There is an isomorphism between $\Diff_\ast(X_g)$ and the group  $\Diff(S_g^1\setminus \calC)$, and therefore a homotopy equivalence $B\Diff_\ast(X_g)\cong B\Diff(S_g^1\setminus \calC)$. On the other hand, by the discussion in Section \ref{sec:cohomlogy-bigmcg} there is a weak homotopy equivalence $B\Diff(S_g^1\setminus \calC)\cong B\Homeo(S_g^1\setminus \calC)$, and the latter is homotopy equivalent to $B\Homeo(S_g^1,\calC)$; see for example \cite[Corollary B.3]{palmer-wu-Documenta}. As $S_g^1$ is a surface of finite type, there is an isomorphism \cite[Theorem 2.3]{CalegariChen2022}
\beq
H_1(B\Homeo(S_g^1,\calC))\cong H_1(B\Homeo(S_g^1)).
\eeq 
Combining this with the fact that $H_1(B\Homeo(S_g^1))=0$ if $g\geq 3$ (c.f. \cite[Lemma 1.1]{HarerSecond}), we obtain $H_1(B\Diff_\ast(X_g))=0$ if $g\geq 3$. Thus $H^2(B\Diff_\ast(X_g))$ is a torsion-free group by the universal coefficient theorem.
Now observe that Theorem \ref{prop:nontrivialXg} gives that $d^\ast:\Z\cong H^2(B\SO(2))\to H^2(B\Diff_\ast(X_g))$ is nontrivial and hence injective.
\end{proof}
\subsection{Surfaces of infinite genus (revisited)}
It may be of independent interest to know that the criterion given by Lemma \ref{lem:nontrivial-euler} can also be applied to  surfaces of infinite genus under the additional assumption that they do not have nonisolated planar ends. 
\begin{prop}\label{prop:avg-fixed-points}
Let $S$ be an orientable surface of infinite genus without nonisolated planar ends, and with one marked point.  Then $\Diff_\ast(S)$ contains finite cyclic subgroups of order $m$, for all $m\geq 1$.
\end{prop}
\begin{proof}
As $S$ has no nonisolated planar ends, the classification of surfaces implies that $S$ is homeomorphic to $Y\setminus\{p_1,p_2,\ldots\}$, where $Y$ is an infinite-genus surface with no planar ends and $\{p_i\}\subset Y$ is a discrete subset of isolated points.

Let $G$ be a finite cyclic group. By \cite[Corollary 1.2]{AugabPatelVlamis2021} and \cite[Theorem 2.8]{epstein-periodic}, the group $G$ acts effectively and by isometries of a complete hyperbolic metric on $Y$. 

To guarantee a fixed point, we take $p\in Y$ with a free orbit which consists of $|G|$ points. Let $G$ act on $S^2$ by rotations around some axis, and take also a point with a free orbit.

Now use these free orbits to perform an equivariant connected sum, which yields a new $G$-surface $S$ obtained from $Y$ by attaching $|G|$-handles. Note that the axis of rotation determines a fixed point of the $G$-action on $S$. Furthermore, this construction does not change the space of ends of $Y$, so it is homeomorphic to $S$.

Finally, observe that repeating this construction at other free orbits of the $G$-action on $Y$, we obtain as many fixed points as needed. Thus we can realize $p_1,p_2,\dots$ as fixed points of a $G$-action on $Y$, and therefore we have an action of $G$ on $Y\setminus\{p_1,p_2,\ldots\}\cong S$  with a fixed point, which is what we wanted to show.
\end{proof}
\section{Morita's non-lifting theorem for infinite type surfaces}\label{sec:non-realization}
In this section we prove Theorem \ref{thm:Nielsen}, which we restate here as follows.
\begin{thm}[Morita’s non-lifting theorem for infinite-type surfaces]\label{thm:nonrealizability} 
Let $S$ be an oriented smooth infinite-type surface either of (possibly infinite) genus  $g\geq 13$ or diffeomorphic to the binary tree surface $\mathrm{Gr}_{\mathfrak{B}}(\Sigma)$ associated to an oriented smooth surface $\Sigma$. Then the extension
\begin{equation*}
    1\rightarrow \Diff^{\id}_*(S)\rightarrow \Diff_*(S)\xto{\pi} \Map_*(S)\rightarrow 1
\end{equation*}
does not split. In fact,  no finite index subgroup of $\Map_*(S)$ 
lifts to $\Diff_*(S)$.
\end{thm}
 Here $\Diff^{\id}_*(S)$ denotes the identity component of $\Diff_\ast(S)$. As pointed out by Morita \cite[Remark 8.2]{morita} the proof below does not work for homeomorphisms.
 \begin{proof}
By considering the identity maps $\Map_\ast^{\delta}(S)\to\Map_\ast(S)$ and $\Diff_\ast^{\delta}(S)\to\Diff_\ast(S)$ and their induced maps in classifying spaces and then cohomology, we obtain the following commutative diagram
\beq
\xymatrix{
H^*(B\Map_\ast(S))\ar[r]^{\pi^*}\ar[d]_{\cong}& H^*(B\Diff_\ast(S))\ar[d]^{B(\id)^*}\\
H^*(B\Map^{\delta}_\ast(S))\ar[r]^{\pi^*}& H^*(B\Diff^{\delta}_\ast(S))
}
\eeq
Suppose, for a contradiction, that the map $\pi:\Diff_\ast(S)\to\Map_\ast(S)$ has a right inverse. So the maps $\pi^*$ in the diagram are injective. Therefore $B(\id)^*$ is injective.  Under our hypothesis, the class $e^4$ is nontrivial rationally.  This follows from Proposition \ref{prop:MMM} when $S$ has finite genus $g\geq 13$, the case when $S$ has infinite genus is a consequence of Proposition \ref{prop:marked-points}, and  Theorem \ref{prop:binary tree surfaces}  implies the case when $S=\mathrm{Gr}_{\mathfrak{B}}(\Sigma)$.  On the other hand, the Bott vanishing theorem \cite{bott-vanishing}, \cite[Theorem 3.24]{morita-libro} can be used in the same way as in \cite[Theorem 8.1]{morita} to show that $e^4\in\ker(B(\id)^*)$. The result follows from this contradiction.
 \end{proof}
\begin{rmk} 
When $S$ is a closed oriented surface of genus $g \geq 6$ or an oriented surface of genus $g\geq 2$ with $k\geq 1$ punctures, another proof of the non-lifting theorem is given in \cite[Theorem 1.2]{bestvina-church-souto}.  It would be interesting to know whether the conclusion of Theorem \ref{thm:nonrealizability} also holds  for the surfaces $X_g$ discussed above, when $1\leq g\leq 12$.
\end{rmk}
\begin{rmk} 
It is worth noting that the ``push'' subgroup $\pi_1(S,\ast)\leq\Map_\ast(S)$ lifts to $\Diff_\ast(S)$ for \textit{any} noncompact surface $S$. This follows from the fact that the homomorphism $\Diff_\ast(S)\to\Map_\ast(S)$ is onto and $\pi_1(S,\ast)$ is a free group. 
\end{rmk}
\appendix
\section{The Nielsen action}\label{sec:Nielsen-action}
We gave in Section \ref{sec:euler-action} an alternative definition of the Euler class via the Nielsen action of $\Map_\ast(S)$ on the circle $S^1$. In this appendix we summarize the aspects of Tappu’s work \cite{Tappu2023} about this action that are relevant in this article. We show that the action is faithful, continuous, and by orientation-preserving homeomorphisms of the circle. Finally, we observe that the existence of such a faithful action provides an obstruction to splitting the Birman exact sequence.

\subsection{Nielsen convex surfaces}
The action of $\Map_\ast(S)$ on orientation-preserving homeomorphisms of the circle is defined for \textit{Nielsen-convex} hyperbolic surfaces. These are complete hyperbolic surfaces which can be constructed by gluing hyperbolic pairs of pants (possibly with cusps) along their boundary components \cite[Theorem  4.5]{AlessandriniEtAl2011}.

It is worth noticing that every smooth surface $S$ without boundary with either negative Euler characteristic or of infinite type admits a complete hyperbolic metric which is Nielsen convex. To see this, take a topological pair of pants decomposition $\mathcal{P}$ of $S$. Replace each pair of pants in $\mathcal{P}$ with a hyperbolic pair of pants having cusps at its punctures, and require that whenever two pairs of pants share a boundary component, the corresponding boundary curves have the same length. Gluing the pieces of $\mathcal{P}$ according to the initial topological decomposition yields a hyperbolic surface homeomorphic to $S$, endowed with a hyperbolic pair of pants decomposition $\mathcal{P}$. As the boundary lengths of the elements of $\mathcal{P}$ can be chosen uniformly bounded above, then $S$ is complete by \cite[Lemma 4.7]{AlessandriniEtAl2011}. Moreover, since $S$ was obtained by gluing hyperbolic pairs of pants, it is Nielsen convex.

\subsection{The action of $\Map_\ast(S)$ on the circle}
Suppose that $S$ is a Nielsen-convex complete hyperbolic surface. Then $S$ is isometric to $\mathbb{H}^2/\Gamma$ where $\Gamma$ is a torsion-free Fuchsian group of the first kind \cite[Proposition 3.1]{Tappu2023}. Denote by $p:\mathbb{H}^2\rightarrow S$ the universal covering projection. Fix $\ast \in X$ and let $\tilde{\ast} \in \mathbb{H}^2$ be a point in $p^{-1}(\ast)$. We identify $\Gamma$ with the fundamental group $\pi_1(S,\ast)$.

A mapping class in $\Map_\ast(S)$ can be represented by a homeomorphism $f:S\rightarrow S$ which fixes $\ast \in S$. This homeomorphism lifts uniquely to a homeomorphism $\widetilde{f}:\mathbb{H}^2 \rightarrow \mathbb{H}^2$  which fixes the point $\tilde{\ast}$.
Moreover $f$ extends to a self-homeomorphism $\widetilde{f}$ of the geodesic boundary $\partial_\infty \mathbb{H}^2\cong S^1$ as we now explain. Let $f_\ast:\Gamma \rightarrow \Gamma$ be the automorphism of the deck group $\Gamma$ induced by $f$. It turns out that $f_\ast$ sends hyperbolic (resp. parabolic) elements to hyperbolic (resp. parabolic) elements \cite[Lemma 9]{Tappu2023}. If $\gamma\in\Gamma$ is a hyperbolic element, we denote by $\gamma^+\in S^1$ the \textit{sink} of $\gamma$. Let $\Gamma_\infty$ be the set of all sinks of hyperbolic elements of $\Gamma$. Since $S$ is complete and Nielsen-convex, the set $\Gamma_\infty$ is dense in $S^1$. The map $$\partial \widetilde{f}: \Gamma_\infty \rightarrow \Gamma_\infty$$ given by $q=\gamma^+\mapsto (f_\ast(\gamma))^+$ is well defined and bijective. Moreover, $\partial \widetilde{f}$ preserves the circular order on $S^1$ \cite[Lemma 12]{Tappu2023}. This fact and the density of $\Gamma_\infty$ in $S^1$ allows one to extend $\partial \widetilde{f}$ uniquely and continuously  over all of $S^1$. Furthermore, this extension is invariant under isotopies, that is, if $f^\prime$ is isotopic to $f$ (relative to $\ast$) then $\partial \widetilde{f^\prime}=\partial \widetilde{f}$.
This defines the map  
\begin{align*}
\rho: \Map_\ast(S) & \rightarrow  \mathrm{Homeo}^+(S^1) \\
[f] & \mapsto  \partial \widetilde{f}
\end{align*} 
where $\mathrm{Homeo}^+(S^1)$ denotes the group of orientation-preserving homeomorphisms of the circle $S^1$. We refer to $\rho$ as the \emph{Nielsen action} of $\Map_\ast(S)$.
\begin{thm}\label{thm:Tappu}
Let $S$ be a complete, Nielsen-convex hyperbolic surface of infinite type. Then the map $\rho$ is a continuous group homomorphism.
\end{thm}
A proof of this result is given in \cite[Proposition 4.1]{Tappu2023}. Although the author does not explicitly state the continuity of the action, his proof of \cite[Lemma 6.5]{Tappu2023} shows that this map is continuous.

As in the case of compact surfaces, the Nielsen action is also faithful for infinite type surfaces. 
\begin{prop}\label{prop:faithful}
    Let $S$ be a complete, Nielsen-convex hyperbolic surface of infinite type. Then the group homomorphism $\rho$ is injective.
\end{prop}
\begin{proof}
Let $f$ an orientation-preserving homeomorphism of $S$ that fixes $\ast\in S$ and represents an element in $\ker (\rho)$. Therefore $\partial \widetilde{f}=\id_{S^1}$, and  if $\gamma\in\Gamma$ is a hyperbolic deck transformation, then both the sink and source points of $\gamma$ are fixed by $\partial \widetilde{f}$. Since $f_\ast:\Gamma\to\Gamma$ preserves primitive elements, it follows that $f_\ast(\gamma)=\gamma$ for every primitive hyperbolic element $\gamma\in\Gamma$. 

Recall that isotopy classes of essential simple closed curves in $S$ correspond bijectively to conjugacy classes of primitive hyperbolic elements of $\Gamma$. Hence $f$ preserves the isotopy class of every simple closed curve in $S$. By the ``Alexander method'' for infinite-type surfaces \cite[Theorem 1.1]{Alexander-InfiniteSurfaces}, we conclude that $f$ is isotopic to the identity.
\end{proof}
\subsection{Non-splitting of the Birman exact sequence for infinite type surfaces}
Let $S$ be an orientable surface of negative Euler charateristic or of infinite type, and let $\ast\in S$ be a marked point. Then the forgetful homomorphism $\Homeo_{\ast}(S)\to \Homeo(S)$ induces the  \emph{Birman exact sequence} (see \cite[Corollaries 1.1 $\&$ 1.3]{Birman1969} and
\cite[Theorem A.2]{Domat2022}) \begin{equation}\label{eqBES}
        1\rightarrow \pi_1(S,\ast)\rightarrow \Map_{\ast}(S)\rightarrow \Map(S)\rightarrow 1.
    \end{equation}

 Since any torsion subgroup of $\mathrm{Homeo}^+(S^1)$ is cyclic, then from Proposition \ref{prop:faithful} any torsion subgroup of $\Map_*(S)$ must be cyclic.
\begin{cor}[Non-splitting of the Birman exact sequence]\label{cor:birman}
 Let $S$ be a complete, Nielsen-convex hyperbolic surface of infinite type. The Birman exact sequence (\ref{eqBES}) does not split whenever $\Map(S)$ contains a non-cyclic finite subgroup. 
\end{cor}
When $g\geq 2$, the group $\Map(S_g)$ contains a finite non-cyclic subgroup, so one can realize it as a group of isometries of $S_g$, and then remove a Cantor set from a neighborhood of a free orbit. This produces surfaces $S = X_g$ for which the Birman exact sequence fails to split.

Another example arises from the fact that the mapping class group of an orientable surface $S$ of infinite genus with no planar ends contains every finite group \cite[Corollary 1.3]{AugabPatelVlamis2021}. In this case, the Birman exact sequence does not split. 
\bibliographystyle{amsalpha}
\bibliography{refs}

\newcommand{\etalchar}[1]{$^{#1}$}
\providecommand{\bysame}{\leavevmode\hbox to3em{\hrulefill}\thinspace}
\providecommand{\MR}{\relax\ifhmode\unskip\space\fi MR }
\providecommand{\MRhref}[2]{%
  \href{http://www.ams.org/mathscinet-getitem?mr=#1}{#2}
}
\providecommand{\href}[2]{#2}
\begin{thebibliography}{HHMV19}

\bibitem[ALP{\etalchar{+}}11]{AlessandriniEtAl2011}
Daniele Alessandrini, Lixin Liu, Athanase Papadopoulos, Weixu Su, and Zongliang
  Sun, \emph{On {F}enchel-{N}ielsen coordinates on {T}eichm\"uller spaces of
  surfaces of infinite type}, Ann. Acad. Sci. Fenn. Math. \textbf{36} (2011),
  no.~2, 621--659. \MR{2865518}

\bibitem[APV21]{AugabPatelVlamis2021}
Tarik Aougab, Priyam Patel, and Nicholas~G. Vlamis, \emph{Isometry groups of
  infinite-genus hyperbolic surfaces}, Math. Ann. \textbf{381} (2021), no.~1-2,
  459--498. \MR{4322618}

\bibitem[BCS13]{bestvina-church-souto}
Mladen Bestvina, Thomas Church, and Juan Souto, \emph{Some groups of mapping
  classes not realized by diffeomorphisms}, Comment. Math. Helv. \textbf{88}
  (2013), no.~1, 205--220. \MR{3008918}

\bibitem[Bir69]{Birman1969}
Joan~S Birman, \emph{Mapping class groups and their relationship to braid
  groups}, Communications on Pure and Applied Mathematics \textbf{22} (1969),
  no.~2, 213--238.

\bibitem[Bol12]{boldsen}
S\o ren~K. Boldsen, \emph{Improved homological stability for the mapping class
  group with integral or twisted coefficients}, Math. Z. \textbf{270} (2012),
  no.~1-2, 297--329.

\bibitem[Bot70]{bott-vanishing}
Raoul Bott, \emph{On a topological obstruction to integrability}, Global
  {A}nalysis ({P}roc. {S}ympos. {P}ure {M}ath., {V}ols. {XIV}, {XV}, {XVI},
  {B}erkeley, {C}alif., 1968), Proc. Sympos. Pure Math., vol. XIV-XVI, Amer.
  Math. Soc., Providence, RI, 1970, pp.~127--131. \MR{266248}

\bibitem[BT01]{Bod-Till}
Carl-Friedrich B\"odigheimer and Ulrike Tillmann, \emph{Stripping and splitting
  decorated mapping class groups}, Cohomological methods in homotopy theory
  ({B}ellaterra, 1998), Progr. Math., vol. 196, Birkh\"auser, Basel, 2001,
  pp.~47--57. \MR{1851247}

\bibitem[CC21]{CalegariChen2021}
Danny Calegari and Lvzhou Chen, \emph{Big mapping class groups and rigidity of
  the simple circle}, Ergodic Theory Dynam. Systems \textbf{41} (2021), no.~7,
  1961–--1987.

\bibitem[CC22]{CalegariChen2022}
\bysame, \emph{Normal subgroups of big mapping class groups}, Trans. Amer.
  Math. Soc. Ser. B \textbf{9} (2022), 957--976. \MR{4498366}

\bibitem[CH24]{chen-he}
Lei Chen and Yan~Mary He, \emph{Non-realizability of some big mapping class
  groups}, Proc. Amer. Math. Soc. \textbf{152} (2024), no.~10, 4503--4514.
  \MR{4806394}

\bibitem[Dom22]{Domat2022}
George Domat, \emph{Big pure mapping class groups are never perfect}, Math.
  Res. Lett. \textbf{29} (2022), no.~3, 691--726, Appendix with Ryan Dickmann.
  \MR{4516036}

\bibitem[Eps81]{epstein-periodic}
D.~B.~A. Epstein, \emph{Pointwise periodic homeomorphisms}, Proc. London Math.
  Soc. (3) \textbf{42} (1981), no.~3, 415--460. \MR{614729}

\bibitem[FM12]{FaMa12}
Benson Farb and Dan Margalit, \emph{A primer on mapping class groups},
  Princeton Mathematical Series, vol.~49, Princeton University Press,
  Princeton, NJ, 2012. \MR{2850125}

\bibitem[FN18]{funar-neretin}
Louis Funar and Yurii Neretin, \emph{Diffeomorphism groups of tame {C}antor
  sets and {T}hompson-like groups}, Compos. Math. \textbf{154} (2018), no.~5,
  1066--1110. \MR{3798595}

\bibitem[Ghy01]{ghys}
\'Etienne Ghys, \emph{Groups acting on the circle}, Enseign. Math. (2)
  \textbf{47} (2001), no.~3-4, 329--407.

\bibitem[Ham21]{hamenstaedt}
Ursula Hamenst\"adt, \emph{Some topological properties of surface bundles},
  Nine mathematical challenges---an elucidation, Proc. Sympos. Pure Math., vol.
  104, Amer. Math. Soc., Providence, RI, [2021] \copyright 2021, pp.~87--105.
  \MR{4337418}

\bibitem[Har83]{HarerSecond}
John Harer, \emph{The second homology group of the mapping class group of an
  orientable surface}, Invent. Math. \textbf{72} (1983), no.~2, 221--239.

\bibitem[HHMV19]{Alexander-InfiniteSurfaces}
Jes\'{u}s Hern\'{a}ndez~Hern\'{a}ndez, Israel Morales, and Ferr\'{a}n Valdez,
  \emph{The {A}lexander method for infinite-type surfaces}, Michigan Math. J.
  \textbf{68} (2019), no.~4, 743--753. \MR{4029627}

\bibitem[Hir76]{hirsch}
Morris~W. Hirsch, \emph{Differential topology}, Graduate Texts in Mathematics,
  vol. No. 33, Springer-Verlag, New York-Heidelberg, 1976. \MR{448362}

\bibitem[JJR21]{rita-jekel}
Solomon Jekel and Rita Jim\'enez~Rolland, \emph{On the non-vanishing of the
  powers of the {E}uler class for mapping class groups}, Arnold Math. J.
  \textbf{7} (2021), no.~1, 159--168. \MR{4238131}

\bibitem[JJR25]{rita-jekel2}
\bysame, \emph{Torsion at the threshold for mapping class groups}, Trans. Amer.
  Math. Soc. (2025).

\bibitem[Mil86]{miller}
Edward~Y. Miller, \emph{The homology of the mapping class group}, J.
  Differential Geom. \textbf{24} (1986), no.~1, 1--14.

\bibitem[Moi77]{moise}
Edwin~E. Moise, \emph{Geometric topology in dimensions {$2$} and {$3$}},
  Graduate Texts in Mathematics, vol. Vol. 47, Springer-Verlag, New
  York-Heidelberg, 1977. \MR{488059}

\bibitem[Mor87]{morita}
Shigeyuki Morita, \emph{Characteristic classes of surface bundles}, Invent.
  Math. \textbf{90} (1987), no.~3, 551--577. \MR{914849}

\bibitem[Mor88]{Morita-bounded}
\bysame, \emph{Characteristic classes of surface bundles and bounded
  cohomology}, A f\^ete of topology, Academic Press, Boston, MA, 1988,
  pp.~233--257. \MR{928403}

\bibitem[Mor01]{morita-libro}
\bysame, \emph{Geometry of characteristic classes}, Translations of
  Mathematical Monographs, vol. 199, American Mathematical Society, Providence,
  RI, 2001, Translated from the 1999 Japanese original, Iwanami Series in
  Modern Mathematics. \MR{1826571}

\bibitem[Mum83]{mumford}
David Mumford, \emph{Towards an enumerative geometry of the moduli space of
  curves}, Arithmetic and geometry, {V}ol. {II}, Progr. Math., vol.~36,
  Birkh\"auser Boston, Boston, MA, 1983, pp.~271--328.

\bibitem[Mun60]{munkres}
James Munkres, \emph{Obstructions to the smoothing of piecewise-differentiable
  homeomorphisms}, Ann. of Math. (2) \textbf{72} (1960), 521--554. \MR{121804}

\bibitem[PW24a]{palmer-wu-Documenta}
Martin Palmer and Xiaolei Wu, \emph{Big mapping class groups with uncountable
  integral homology}, Doc. Math. \textbf{29} (2024), no.~1, 159--189.

\bibitem[PW24b]{palmer-wu-homology}
\bysame, \emph{On the homology of big mapping class groups}, J. Topol.
  \textbf{17} (2024), no.~4, Paper No. e12358, 41. \MR{4806781}

\bibitem[PW25]{palmer-wu-JLMS}
\bysame, \emph{Compact and finite-type support in the homology of big mapping
  class groups}, to appear in J. Lond. Math. Soc. (2025).

\bibitem[Ric63]{richards}
Ian Richards, \emph{On the classification of noncompact surfaces}, Trans. Amer.
  Math. Soc. \textbf{106} (1963), 259--269. \MR{143186}

\bibitem[Tap24]{Tappu2023}
Chaitanya Tappu, \emph{A moduli space of marked hyperbolic structures for big
  surfaces}, preprint
  \url{https://pi.math.cornell.edu/~tappu/research/marked_moduli_space.pdf}
  (2024).

\bibitem[Thu86]{Thurston}
William~P. Thurston, \emph{Earthquakes in two-dimensional hyperbolic geometry},
  Low-dimensional topology and {K}leinian groups ({C}oventry/{D}urham, 1984),
  London Math. Soc. Lecture Note Ser., vol. 112, Cambridge Univ. Press,
  Cambridge, 1986, pp.~91--112. \MR{903860}

\bibitem[vK23]{kerekjarto}
B.~von Ker{\'e}kj{\'a}rt{\'o}, \emph{Vorlesungen {\"u}ber {Topologie}. {I}.:
  {Fl{\"a}chentopologie}. {Mit} 80 {Textfiguren}.}, Grundlehren Math. Wiss.,
  vol.~8, Springer, Cham, 1923 (German).

\bibitem[Yag00]{yagasaki-homeo}
Tatsuhiko Yagasaki, \emph{Homotopy types of homeomorphism groups of noncompact
  {$2$}-manifolds}, Topology Appl. \textbf{108} (2000), no.~2, 123--136.
  \MR{1787857}

\bibitem[Yag02]{yagasaki-diffeo}
\bysame, \emph{Homotopy types of diffeomorphism groups of noncompact
  2-manifolds}, no. 1248, 2002, General geometric topology and its applications
  (Japanese) (Kyoto, 2001), pp.~50--58. \MR{1924121}

\end{thebibliography}

Mauricio Bustamante\\
Departamento de Matem\'aticas, Pontificia Universidad Cat\'olica de Chile\\
\texttt{mauricio.bustamante@uc.cl}

Rita Jim\'enez Rolland\\
{ Instituto de Matemáticas, Universidad Nacional Autónoma de México}\\
\texttt{rita@im.unam.mx}

Israel Morales\\
{Departamento de Matem\'atica y Estad\'istica, Universidad de la Frontera}\\
\texttt{israel.morales@ufrontera.cl}
\end{document}